\documentclass[12pt]{article}
\usepackage{amsmath,amssymb, amsfonts, amsthm,amscd, lineno}
\usepackage[T2A]{fontenc}
\usepackage[cp1251]{inputenc}
\usepackage[english]{babel}
\usepackage{graphicx}

\usepackage{xcolor}

\newtheorem{theorem}{Theorem}
\newtheorem{lemma}{Lemma}
\newtheorem{proposition}{Proposition}
\newtheorem{corollary}{Corollary}
\theoremstyle{definition}

\DeclareMathOperator{\diag}{diag}

\begin{document}

\begin{center}
{\huge Reduction of  matrices over simple Ore domains}
\end{center}
\vskip 0.1cm \centerline{{\Large Victor Bovdi, Bohdan Zabavsky}}

\vskip 0.3cm

\footnotesize{\noindent\textbf{Corresponding author}: V. Bovdi
\vskip 0.3cm
\footnotesize{\noindent\textbf{Address:} \textit{UAEU, Al Ain, United Arab Emirates (V.~Bovdi)
\newline
Ivan Franko National University, Lviv, Ukraine (B.~Zabavsky)}
\vskip 0.3cm
\footnotesize{\noindent\textbf{Abstract:} \textit{We study the theory of diagonal reductions of matrices over simple Ore domains of finite stable range.    We cover the cases of  $2$ -simple rings of stable range 1, Ore domains and certain cases of Bezout domains.}

\vskip 0.3cm
\footnotesize{\noindent\textbf{Keywords and phrases:} \textit{Ore domain, B\'ezout domain, stable range, $n$-simple ring, diagonal reduction of a matrix.}}
\vskip 0.3cm

\noindent{\textbf{Mathematics Subject Classification}}: 19B10, 16E50, 16U10, 16U20}}

\normalsize
\section{Introduction and results}
The problem of diagonalization of matrices over rings is a classical problem of  ring theory. An overview  can be found in  \cite{1zabava2012}. While commutative elementary divisor rings have been investigated fairly systematically, noncommutative elementary divisor rings have not received such  attention.  Nevertheless  significant results have been obtained in this field. For example,  Henriksen \cite[Theorem 3, p.\,134]{2henriks1973} showed that any matrix over an unit-regular ring can be reduced to a diagonal form by multiplications from  left and  right  by  invertible matrices of  suitable  sizes.

According to Cohn \cite[Theorem 3.6, p.\,255]{3cohn1987},   a right principal B\'ezout domain has the reduction  matrix  property, at least when  certain  conditions on the diagonal elements of its  diagonal form are satisfied.  An example of such B\'ezout domain was constructed in  \cite[Lemma, p.\,27]{4cohn1985}; is should be noted that it  is also an example for  simple B\'ezout domains, i.e. domains with trivial two-sided ideals only.

The study of the connections between  the stable range of a ring and the   reduction of  matrices over that   ring showed (see for example   \cite[Theorem 4.4.1, p.\,185]{1zabava2012})  that a simple B\'ezout domain was an elementary divisor ring if and only if it was a $2$-simple domain.

The notion of a stable range of a ring was introduced  in algebraic $K$-theory and   has been  proved  useful for the study of certain  problems in the ring theory. In particular, it was proved that the  stable range of an elementary divisor ring did not exceed 2 \cite[Theorem 1.2.40, p.\,48]{1zabava2012} and each B\'ezout domain is a Hermite ring \cite{Amitsur}.
Several important results about connections between B\'ezout domains, Hermite rings, stable range and  elementary divisor rings were obtain in the papers of Amitsur, Ara, Goodearl, Menal, Moncasi,  O’Meara, Paphael and others (see for example
\cite{Amitsur, ara2, menal, 1zabava2012}).  In the present  paper we study the   diagonal reduction of matrices over a simple Ore domain. This investigation reveals a  connection to  the theory of  full matrices over certain classes of rings.

Our main results are the following.
\begin{theorem}\label{T:1}
Let $R$ be a $2$-simple ring of stable range 1 and let  $a, b\in R$ be such that either $ab\ne 0$ or $ba\ne 0$.
The matrix   $\text{diag}(a,b)\in R^{2\times 2}$ can be reduced to the form  $\text{diag}(1,c)\in R^{2\times 2}$ for some   $c\in R$.
\end{theorem}

In the case of $(n+1)$-simple domains (where $n\ge2$) we have the following.

\begin{theorem}\label{T:3}
Let $R$ be an $(n+1)$-simple Ore domain of stable range $n\in \mathbb{N}$. For each non-zero divisor  $A \in R^{(n+1) \times (n+1)}$ there  exist  $P, Q \in GL_{n+1}(R)$ and $A_0\in R^{n\times n}$ such that
 \begin{equation}\label{URAW:1}
    PAQ=
\left(\!\begin{smallmatrix}
          1 & 0 \,.\,.\,. \, 0 \\
\begin{smallmatrix}
 0\\[-5pt]\vdots\\0
\end{smallmatrix}
 & {\displaystyle A_0} \\
                \end{smallmatrix}\right).
\end{equation}
\end{theorem}

As a consequence of Theorem~\ref{T:3} we have the following result.

\begin{theorem}\label{T:3b}
Let $R$ be a 2-simple Ore domain of stable range 1. For
each non-zero divisor matrix $A \in R^{2 \times 2}$  there  exist  $P, Q \in GL_{2}(R)$ such
that
\[
PAQ=\begin{pmatrix}
        1 & 0 \\
        0 & a
      \end{pmatrix}\quad \text{for some}\quad   a\in R.
\]
\end{theorem}

Since each B\'ezout domain is an  Ore domain \cite[Corollary 2.1.1, p.\,53]{1zabava2012} and \cite{menal},   from  Theorem~\ref{T:3}   we have the following.

\begin{theorem}\label{T:4}
Let $R$ be a $n$-simple B\'ezout  domain. If $m\ge n\in \mathbb{N}$, then  for each  $A\in R^{m\times m}$    there   exist  $P, Q\in GL_{m}(R)$ such that
$
PAQ=\left(\begin{smallmatrix}
E& 0\\
0& T\\
0&0
\end{smallmatrix}\right)$,
where $E$ is the  identity matrix, $T=\text{diag}(A_1,\ldots,A_k)$ and each   $A_i\in R^{n \times n}$ is a triangular matrix.
\end{theorem}

\section{Notations  and Preliminary Results}
The set of positive integers is denoted by $\mathbb{N}$.
Let $R$ be an   associative ring with  nonzero unit and let $r,s,n\in \mathbb{N}$.  The  vector space of matrices over the ring $R$ of size ${r\times s}$  is denoted by $R^{r\times s}$.  Groups of units of the rings $R$ and $R^{n \times n}$   are denoted by $U(R)$ and $GL_n(R)$, respectively.

A ring $R$ is called  {\it right (left) B\'ezout ring} if each   finitely generated right (left) ideal of $R$  is principal. A ring $R$ which  is simultaneously  right  and  left B\'ezout ring is called  {\it B\'ezout ring}. A domain $R$ is called  {\it right (left) Ore domain} if for each  $a,b\in R\setminus\{0\}$ one has $aR\cap bR\neq \{0\}$  ($Ra\cap Rb\neq \{0\}$). Each  Ore domain is a domain that is simultaneously  right  and  left Ore domain and each B\'ezout domain is an Ore domain \cite[Proposition 1.8, p.\,53]{5stenstrom1971}.

Each  Ore domain $R$ can be embedded into  a division ring \cite[Proposition 5.2, p.\,259]{3cohn1987}, so  we can define ranks of a  matrix $A$ over $R$  on their  rows  $\rho_r(A)$ and their columns $\rho_c(A)$, respectively. Note that, the  numbers   $\rho_r(A)$ and $\rho_c(A)$ do not change under  elementary transformations of $A$.

The smallest  $m\in \mathbb{N}$ such that a matrix $A\in R^{r\times s}$ is a product of two matrices of size $r\times m$  and $m\times s$,  is called the {\it inner rank $\rho(A)$ } of $A$ \cite[p.\,244]{Cohn_book}). Note that $\rho(A)\leq \min\{\rho_r(A), \rho_c(A)\}$ and the number   $\rho(A)$ does not change under elementary transformations (see \cite[p.\,244]{Cohn_book}). If $R$ is a right B\'ezout domain, then $\rho(A)=\rho_r(A)=\rho_c(A)$ for any $A$ over $R$. A  matrix $A\in R^{n\times n}$  is called {\it full} if $\rho(A)=n$. Note that  $A\in R^{r\times s}$ is a left zero divisor in $R^{r\times s}$  if and only if $\rho_c(A)<r$  (see \cite[Collorary, p.\,245]{Cohn_book}). A square matrix $A$ over a right Ore domain is not a left zero-divisor if and only if it is a full matrix (see \cite[Proposition 5.2]{Cohn_book}).

The following  important result holds for  full matrices over $FI$-ring.

\begin{proposition}\label{P:1}\cite[Theorem 6.4]{Cohn_book}
If $R$ is  an  $FI$-ring
then  $R^{n \times n}$ is a ring with unique factorization of full matrices, i.e. for any  full matrix  is either an  invertible matrix or  is a product atoms and any two decompositions a full matrix are isomorphic.
\end{proposition}

A  matrix $A=(a_{ij})\in R^{n\times m}$ is called  {\it diagonal} if $a_{ij}=0$ for all $i\ne j$ and we write  it as $\diag(a_{11},\ldots,a_{nn})$.  Two matrices $A$ and $B$ over a ring $R$ is equivalent if there  exist invertible matrices $P$ and $Q$ over  $R$ such that $B=PAQ$. If a matrix $A$ over $R$ is equivalent to a diagonal matrix $D=(d_{ii})$ with the property that $d_{ii}$ is a total divisor of $d_{i+1,i+1}$ (i.e. $Rd_{i+1,i+1}R\subset d_iR\cap Rd_i$), then  we say that $A$  admits a canonical diagonal reduction. A ring $R$ over which every  matrix admits a  canonical diagonal reduction is called an {\it elementary divisor ring}.

A ring $R$ is called {\it  right (left) Hermite} if each matrix $A\in R^{1\times 2}$ ($A\in R^{2\times 1}$) admits a diagonal reduction.  A ring which is  right and  left  Hermite is called a {\it Hermite ring}. Moreover, each  elementary divisor ring is Hermite, and a right (left) Hermite ring is a right (left) B\'ezout ring \cite[p.\,298--299]{1zabava2012}.

A row $(a_1,\ldots, a_n)\in R^{n}$  is called {\it unimodular} if $a_1R+a_2R+\cdots+a_nR=R$.
An unimodular $n$-row $(a_1,\ldots, a_n)\in R^{n}$   over a ring $R$  is called {\it reducible} if there exist a $(n-1)$-row  $(b_1,\ldots, b_{n-1})\in R^{n-1}$ such that
\[
(a_1+a_nb_1,a_2+a_nb_2,\ldots, a_{n-1}+a_nb_{n-1})\in R^{n-1}
\]
is unimodular. If $n\in\mathbb{N}$ is the smallest number  such that any  unimodular $(n+1)$-row is reducible, then  $R$ has {\it stable range} $n$, where $n\geq 2$. A ring $R$ has  {\it stable range} 1 if $aR + bR = R$ implies that $(a + bt)R = R$ for some $t\in R$.

Let $R$ be a simple ring. Clearly, $RaR=R$ for each  $a\in R\setminus\{0\}$ and there  exist $n\in \mathbb{N}$ and  $u_1,  \ldots, u_n, v_1, \ldots, v_n\in R$ such that
\begin{equation}\label{URAW:2}
u_1av_1+u_2av_2+\dots+u_nav_n=1.
\end{equation}
If for all $a \in R\setminus\{0\}$, there exists a minimal  $n \in \mathbb{N}$ which satisfies  \eqref{URAW:2}, then  $R$ is called {\it $n$-simple} ring.

A ring with identity $R$ is called {\it unit-regular} if for every  $a\in R$ there is a unit $u\in U(R)$ with $a = aua$. A von Neumann regular ring  $R$ is  unit-regular if and only if $R$ has stable range 1.

In the  sequel we use freely the following results:
\begin{proposition}\label{P:2} The following conditions hold:
\begin{itemize}
\item[(i)]  \cite[Theorem 3]{2henriks1973}   Each  $2$-simple unit-regular ring is an elementary divisor ring;
\item[(ii)]\cite[Proposition 1.8]{5stenstrom1971} Each right B\'ezout domain is a right Ore domain;
\item[(iii)] \cite[Corollary 2.1.2, p.\,56]{1zabava2012}  Each right (left) Hermite ring is a ring of stable range 2;
\item[(iv)] \cite[Clorollary 2.1.5, p.\,60]{1zabava2012} If $R$ is a right B\'ezout ring of finite stable range $n\in \mathbb{N}$, then  each  right (left) unimodular row (column) of length $m\ge n+1$ is completive to an element of the subgroup $GE_m(R)$  of elementary matrices of the group $GL_m(R)$, with the addition of  extra columns and rows.
\end{itemize}
\end{proposition}

\section{Proofs}
We start our proof with the following.

\begin{lemma}\label{L:1}
Let $R$ be an $n$-simple ring. For any  $a_1, \ldots, a_n \in R$ with the property  $a_1  \cdots a_n \neq 0$,    there exist  $u_1,\ldots, u_n, v_1,\ldots, v_n \in R$  such that
\begin{equation}\label{URAW:3}
\sum_{i=1}^n u_ia_iv_i=1.
\end{equation}
\end{lemma}
\begin{proof}
Since $a_1\cdots a_n\ne0$ and $R$ is  $n$-simple,
$\sum_{i=1}^n x_i(a_1\cdots a_n)y_i=1$
for some $x_1,\ldots,x_n, y_1,\ldots,y_n\in R$ by \eqref{URAW:2}.  Put  $u_1:=x_1$,\;  $u_2:=x_2a_2$, \ldots, $u_n:=x_na_1\cdots a_{n-1}$, \; $v_1:=a_2\cdots a_ny_1$, $v_2:=a_3\cdots a_ny_2$, \ldots, $v_n:=y_n$. Obviously,  \eqref{URAW:3} is a consequence of the equation  $\sum_{i=1}^n x_ia_iy_i=1$.
\end{proof}

\begin{lemma}\label{L:2}
Let $R$ be a simple elementary divisor ring. For each  $a\in R\setminus\{0\}$ there exist $u_1, u_2, v_1, v_2 \in R$ such that
\[
\text{either}\quad  u_1av_1+u_2av_2=1
\quad\text{or}\quad
u_1(1-a)v_1+u_2(1-a)v_2=1.
\]
\end{lemma}
\begin{proof}
If $a\in R\setminus\{0\}$, then
$\diag(a,a)\cdot P=Q\cdot \diag(z,b)$  (see the definition of $R$), in which  $z,b\in R$, $P=(p_{ij}), Q=(q_{ij})\in GL_2(R)$ and $RbR\subseteq zR\cap Rz$. Hence
\begin{equation}\label{URAW:4}
    \begin{pmatrix}
      ap_{11}& ap_{12} \\
      ap_{21} & ap_{22}
    \end{pmatrix} =\begin{pmatrix}
      q_{11}z& q_{12}b \\
      q_{21}z & q_{22}b
    \end{pmatrix}.
    \end{equation}
Since $R$ is simple, either  $z\in U(R)$  \quad or \quad   $b=0$.

\bigskip
Consider each case separately.

Case 1. Let $z\in U(R)$. We can assume $z:=1$, so from  \eqref{URAW:4} we have
\begin{equation}\label{URAW:5}
ap_{11}=q_{11}\quad\text{and}\quad  ap_{21}=q_{21}.
\end{equation}
Since  $\begin{pmatrix}
     q_{11} \\
      q_{21}
\end{pmatrix}\not=0$ as the  first column of  $Q\in GL_2(R)$,  $R=Rq_{11}+Rq_{21}$. This yields\quad   $1=uq_{11}+vq_{21}:=u_1av_1+u_2av_2$\quad  for some  $u,v\in R$  by \eqref{URAW:5}, where $u_1:=u$, $u_2:=v$, $v_1:=p_{11}$ and  $v_2:=p_{21}$.

Case 2. Let  $b=0$. Clearly $\begin{pmatrix}
      p_{12} \\
      p_{22}
\end{pmatrix}\not=0$ as the  second  column of  $P\in GL_2(R)$ and   $ap_{12}=ap_{22}=0$ by  \eqref{URAW:4}, so
\begin{equation}\label{URAW:6}
p_{12}=(1-a)p_{12}\quad\text{and}\quad  p_{22}=(1-a)p_{22}.
\end{equation}
As in the previous case,\quad   $1=xp_{12}+yp_{22}:=u_1(1-a)v_1+u_2(1-a)v_2$ \quad  for some $x,y\in R$ by \eqref{URAW:6}, where  $u_1:=x$, $u_2:=y$, $v_1:=p_{12}$ and  $v_1:=p_{22}$.
\end{proof}

\begin{corollary}\label{C-1}
Let $R$ be a simple elementary divisor ring. For each  $a\in R\setminus\{0\}$
there exist $L, M\in GL_2(R)$ and $b\in R$, such that
\[
  L\cdot \diag(a,a)\cdot M=\diag(1,b).
\]
Moreover, a simple elementary divisor domain is a  $2$-simple domain.
\end{corollary}
\begin{proof}
According to  the proof of Lemma \ref{L:2}, if $\diag(a,a)\cdot P=Q\cdot  \diag(z,0)$ then $\diag(1-a,1-a)\cdot T=S\cdot  \diag(1,c)$, where   $P,Q,T,S\in GL_2(R)$ and  $c\in R$. Conversely,  if $\diag(1-a,1-a)\cdot  S=T \cdot  \diag(t,c)$, where   $T,S\in GL_2(R)$ and  $t, c\in R$,  then $\diag(a,a)\cdot P=Q\cdot  \diag(1,b)$ for some  $P,Q\in GL_2(R)$.

If $R$ is a simple elementary divisor domain, then  for each   $a\in R\setminus\{0\}$ there  exist  $u_1,u_2, v_1,v_2\in R$,  such that $u_1av_1+u_2av_2=1$ by Lemma \ref{L:2}. The case 2 (see the proof of Lemma \ref{L:2}) is impossible for  $a\in R\setminus\{0\}$, so
\[
  \diag(a,a)\cdot P=Q\cdot  \diag(1,b), \qquad (P,Q\in GL_2(R), b\in R).
\]
\end{proof}

The concept of an $(n+1)$-simple ring closely linked to the theory of rings of stable range $n$. First consider $2$-simple rings of stable range 1.

\begin{lemma}\label{L:3}
Let $R$ be a $2$-simple ring of stable range 1. For each $a\in R\setminus\{0\}$ there  exist $x,y\in R$ such that $a+xay\in U(R)$.
\end{lemma}
\begin{proof}
Since $R$ is  $2$-simple, for each  $a\in R\backslash\{0\}$ there exist $u_1, u_2, v_1, v_2\in R$ such that $u_1av_1+u_2av_2=1$, so  $u_1aR+u_2aR=R$. The ring  $R$ has  stable range 1, so $(u_2a+u_1at)R=R$ for some $t\in R$ and
\begin{equation}\label{URAW:7}
u_1at+u_2a=w_1\in U(R).
\end{equation}
Similarly,   $u_1aR+u_2R=R$ and      $u_1as+u_2=w_2\in U(R)$ for some  $s\in R$, so   $u_2=w_2-u_1as$. From  \eqref{URAW:7} we obtain
\[
u_1at+w_2a-u_1asa=u_1a(t-sa)+w_2a=w_1\in U(R)
\]
and $xay+a=x\in U(R)$, in which  $x:=w_2^{-1}u_1$ and  $y:=t-sa$.
\end{proof}

As a consequence of Lemma \ref{L:3}, we have the following.

\begin{lemma}\label{L:4}
Let $R$ be a $2$-simple ring of stable range 1. Each matrix   $0\neq\text{diag}(a,a)\in R^{ 2 \times 2}$ can be reduce to $\text{diag}(1,b)\in R^{2\times 2}$ for some   $b\in R$.
\end{lemma}
\begin{proof}
Let $a\not=0$. Since $xay+a=u\in U(R)$  for some  $x,y\in R$  by Lemma  \ref{URAW:6},
  \[
  \begin{pmatrix}
      u^{-1}& 0 \\
      0 & 1
    \end{pmatrix}
    \cdot
  \left[
    \begin{pmatrix}
      1& x \\
      0 & 1
    \end{pmatrix}\begin{pmatrix}
      a& 0 \\
      0 & a
    \end{pmatrix}\begin{pmatrix}
      1& 0 \\
      y & 1
    \end{pmatrix}
    \right]
    =\begin{pmatrix}
      1& u^{-1}xa \\
      ay & a
    \end{pmatrix}
\]
and  $
S\begin{pmatrix}
      1& u^{-1}xa \\
      ay & a
    \end{pmatrix}\cdot T=
    \diag(1,b)$ for some  $b\in R$ and   $S,T\in GL_2(R)$.
\end{proof}

\begin{proof}[Proof of Theorem \ref{T:1}]
For any   $a,b\in R$ with  $ab\not=0$ there exist $u_1, u_2, v_1, v_2\in R$ such that $u_1av_1+u_2bv_2=1$ by Lemma \ref{L:1}. Hence  $u_1aR+u_2bR=R$ and  $(u_2b+u_1at)R=R=(u_2+u_1as)R$ for some  $s, t\in R$, because  $R$ has  stable range 1. Consequently, $u_1at+u_2b=w_1\in U(R)$  and  $u_1as+u_2=w_2\in U(R)$.  It follows that
\[
w_1=u_1at+(w_2-u_1as)b=u_1a(t-sb)+w_2b\in U(R)
\]
and $xay+b=w_2^{-1}w_1\in U(R)$, in which  $x:=w_2^{-1}u_1$ and  $y:=t-sb$.
Finally,
\[
\begin{pmatrix}
      w_1^{-1}w_2& 0 \\
      0 & 1
    \end{pmatrix}
    \cdot
\left[
    \begin{pmatrix}
      x& 1 \\
      1 & 0
    \end{pmatrix}\begin{pmatrix}
      a& 0 \\
      0 & b
    \end{pmatrix}\begin{pmatrix}
      y& 1 \\
      1 & 0
    \end{pmatrix}
\right]
    =\begin{pmatrix}
      1& w_1^{-1}w_2xa \\
      ay & a
    \end{pmatrix}
    \]
and  $
S\cdot \begin{pmatrix}
      1& w_1^{-1}w_2xa \\
      ay & a
    \end{pmatrix}\cdot T=
    \diag(1,c)$ for some  $c\in R$ and   $S,T\in GL_2(R)$.

The  case when  $ab=0$, but $ba\ne0$ can be  treated similarly.
\end{proof}

\begin{lemma}\label{L:5}
A simple unit-regular ring $R$ is an elementary divisor ring if and only if for each  idempotent $e\in R$, there exist  $u_1, u_2, v_1, v_2\in R$ such that
\[
u_1ev_1+u_2ev_2=1.
\]
\end{lemma}
\begin{proof}
This  is a simple consequence of Lemma \ref{L:2}.
\end{proof}

\begin{proof}[Proof of Theorem \ref{T:3}]
According to the restrictions imposed on $R$ and  $A$ we have $AT=\diag(\varepsilon_1,\ldots, \varepsilon_{n+1})\not=0$  for some matrix $T$. Since $R$ is a $(n+1)$-simple domain, we have
\[
u_1\varepsilon_1v_1+u_2\varepsilon_2v_2+\dots+u_{n+1}\varepsilon_{n+1}v_{n+1}=0
\]
for some $u_1,\ldots, u_{n+1}, v_1,\ldots, v_{n+1}\in R$.
It follows that
\begin{equation}\label{URAW:8}
(u_1\, u_2, \ldots,  u_{n+1})\cdot A\cdot
(      w_1\, w_2, \ldots,  w_{n+1})^T=1,
    \end{equation}
in which
$(w_1,\ldots, w_{n+1})^T=T^T\cdot (u_1, \ldots , u_{n+1})^T$ and where ${(w_1,\ldots, w_{n+1})}^T$ is the transposed matrix of $(w_1,\ldots, w_{n+1})$. This yields
\[
u_1R+\cdots+u_{n+1}R=Rw_1+\cdots+Rw_{n+1}=R.
\]
Since $R$ has  stable range $n$,  the row $(u_1, \ldots,  u_{n+1})$ and the column $(w_1, \ldots,  w_{n+1})^T$ can  be completed (see Proposition  \ref{P:2}(iv)) to the following matrices
 \[
U=\left(\!\begin{smallmatrix}
          u_1&\cdots & u_{n+1} \\
             & &  \\
          \begin{smallmatrix}
          &&\\
          \end{smallmatrix}
 & {\displaystyle \star} \\
                \end{smallmatrix}\right),\quad
W=\left(\!\begin{smallmatrix}
          & \\
\begin{smallmatrix}
 w_1\\
 \vdots\\
 w_{n+1}\\
\end{smallmatrix}
 & {\displaystyle \star} \\
                \end{smallmatrix}\right) \in GE_{n+1}(R).
\]
Finally,  $UAW=
    \left(\!\begin{smallmatrix}
          1 & * \,\cdots \, * \\
\begin{smallmatrix}
 *\\[-5pt]\vdots \\ *\\
\end{smallmatrix}
 & {\displaystyle B} \\
                \end{smallmatrix}\right)$ by \eqref{URAW:8} and
using  elementary transformations of rows and columns it can be transformed to the  form \eqref{URAW:2}.\end{proof}

\begin{proof}[Proof of Theorem \ref{T:4}]    The fact that each   $A_i\in R^{n \times n}$ is a triangular matrix  follows from the fact that  each commutative B\'ezout domain is a Hermite ring see \cite[p.\,29-30]{1zabava2012}.
\end{proof}

\end{document}